\documentclass[reqno,12pt]{amsproc}

\usepackage{amsmath}
\newtheorem{theo}{Theorem}
\newtheorem{rem}{Remark}

\newtheorem{lem}{Lemma}

\newtheorem{df}{Definition}

\newcommand\eps\varepsilon
\newcommand\ph\varphi
\newcommand\kap\varkappa

\newcommand\bs {\boldsymbol}

\usepackage{hyperref}
\usepackage{url}

\usepackage{enumerate}
\usepackage{graphicx}
\usepackage{tikz-cd}
\usepackage{float}
\usepackage{yfonts}
\usepackage{mathrsfs}
\usepackage{amsthm}
\usepackage{amscd}

\begin{document}\title
{The D'Alembert-Lagrange Principle and Lagrange Equations}
\author[Oleg Zubelevich]{Oleg Zubelevich\\Faculty of Mechanics and Mathematics,\\
Lomonosov Moscow State University,\\ 1 Leninskie Gory, Moscow 119991, Russia}
\email{oezubel@gmail.com}

\date{}

\begin{abstract}In this expository-research article we present a rigorous geometric framework for analytical mechanics in the extended phase space, accommodating general constraints that may be non-linear in velocities. We prove that the space of virtual displacements and ideal constraint reactions depend strictly on the intrinsic geometry of the constraint manifold and are entirely invariant under the choice of analytical constraint equations. Furthermore, we show that the Lagrange equations of the second kind are  derived through the covector covariance of the variational derivative operator under the embedding mapping of the configuration manifold.
\end{abstract}

\maketitle

\section{Introduction}

 At the core of analytical mechanics lie the D'Alembert--Lagrange principle and the Lagrange equations, which offer a powerful alternative to Newton's second law when dealing with constrained mechanical systems. Traditionally, introductory treatments of constrained dynamics rely heavily on physical intuition, coordinate-dependent derivations, and infinitesimal virtual displacements (see, e.g., classical texts by Gantmacher \cite{gantmacher1970} and Rosenberg \cite{rosenberg1977}). However, a modern mathematical physics formulation demands a coordinate-free, geometrically rigorous framework \cite{arnold2006mathematical, bloch2015nonholonomic, marsden1999}.

A primary objective of this expository paper is to present a clean, mathematically precise, and self-contained derivation of the D'Alembert--Lagrange principle and Lagrange equations. In doing so, this work explicitly emphasizes two fundamental aspects that are often treated heuristically or omitted in standard literature.

In modern geometric mechanics (see, e.g., Vershik \& Faddeev \cite{vershik1972, vershik1981}, Marle \cite{marle1998}, Lewis \& Murray \cite{lewis1995}, and Bloch \cite{bloch2015nonholonomic}), constraints are defined directly in purely geometric terms as a distribution on the configuration space, leading to constraint equations that are linear in velocities. The formulation presented in this paper takes a different starting point that bridges classical analytical mechanics and geometric mechanics.
Rather than presupposing a subbundle of the tangent bundle, we begin with explicit analytical constraint equations $\ph(t,x,\dot{x}) = 0$, which allows for arbitrary (and potentially non-linear \cite{rumyantsev1984}) velocity dependencies. Under this setup, the constraint space is treated as a smooth submanifold  in the extended phase space. We then show that the analytically derived space of virtual displacements  and reaction forces do not depend on the specific choice of the constraint functions $\ph$, thereby recovering full geometric invariance without restricting the system to linear velocity constraints.

Another key methodological feature of this paper is the derivation of the Lagrange equations of the second kind for holonomic constraints. Rather than applying traditional coordinate transformations, we exploit covariance of the variational derivative operator
under an embedding mapping. This geometric perspective highlights the covector field nature of the variational derivative, allowing the Lagrange equations to be obtained elegantly via a pullback to the configuration manifold.

The article is organized as follows.

In Section \ref{1}, we briefly recall the physical formulation of the problem concerning the motion of a mechanical system with constraints.

In Section \ref{2}, we prove a theorem stating that the forces ensuring the motion of the system consistent with ideal constraints exist and are uniquely determined by the constraint equations. Thus, Lagrange's equations of the first kind are introduced.

In classical textbooks on theoretical mechanics, this theorem is introduced without proof and is referred to as the principle of release from constraints or the axiom of constraints.

In this context, we interpret the constraint equations as first integrals of Lagrange's equations with multipliers.

Furthermore, we prove theorems demonstrating a one-to-one correspondence between the motions derived from the d'Alembert–Lagrange principle and those derived from Lagrange's equations with multipliers.

In Section \ref{sdf55}, we demonstrate that the reaction forces of ideal constraints are determined by the geometry of the manifold that specifies these constraints. To the author's knowledge, this property has not been previously noted in the literature.

Since the transition from a physical problem to a mathematical model allows for different choices of constraint equations (for instance, the same  line in the $XY$ plane can be defined by the equation $y = 2x$ or by $e^{y-2x} - 1 = 0$), this property of ideal constraint reactions is  of methodological importance.

In Section \ref{5}, we present the standard definition of holonomic and nonholonomic constraints.

In Section \ref{6}, we prove a version of the theorem on the covariance of the variational derivative. Since this result is of interest in its own right, and particularly in nonholonomic mechanics, we verify it directly without reference to a variational problem. Furthermore, we prove this statement in a somewhat more general form than is customary.

While the standard formulation asserts that the variational derivative behaves as a covector under coordinate transformations on the configuration manifold, we show that it continues to behave as a covector under restriction to a submanifold.

In Section \ref{cbbbbb}, we derive Lagrange's equations of the second kind from the d'Alembert–Lagrange principle.

Unlike standard textbooks, our derivation is based on the covariance of the variational derivative established in Section \ref{6}, rather than on the somewhat archaic procedure dating back to Lagrange.

We establish a one-to-one correspondence between the motions obtained from the d'Alembert–Lagrange principle and the solutions to Lagrange's equations of the second kind.

In Section \ref{8}, we discuss how the terminology introduced here corresponds to the classical one.

In Section \ref{9}, we prove a technical theorem required for Section \ref{sdf55}.

\section{Physical Description of the Problem}\label{1}

Let a system of particles with masses
$m_1,\ldots,m_\nu$ be given.
Let $\boldsymbol r_1,\ldots,\boldsymbol r_\nu$ denote the position vectors of these particles in some inertial Cartesian frame $OXYZ$:
$$\boldsymbol r_i=(X_i,Y_i,Z_i).$$
We represent the force acting on the $i$-th particle as the sum
$\boldsymbol F_i+\boldsymbol R_i,$ where
\begin{align}\boldsymbol F_i&=\boldsymbol F_i(t,\boldsymbol r_1,\ldots,\boldsymbol r_{\nu},\boldsymbol {\dot r}_1,\ldots,\boldsymbol {\dot r}_{\nu})=(F_i^x,F_i^y,F_i^z),\nonumber\\
 \boldsymbol R_i&=\boldsymbol R_i(t,\boldsymbol r_1,\ldots,\boldsymbol r_{\nu},\boldsymbol {\dot r}_1,\ldots,\boldsymbol {\dot r}_{\nu})=(R_i^x,R_i^y,R_i^z).\nonumber
 \end{align}
The forces $\boldsymbol F_i$ are assumed to be given and are called active forces. 

The forces $\boldsymbol R_i$ compel the system of particles to move in accordance with the equations
$$\Phi^k(t,\boldsymbol r_1,\ldots,\boldsymbol r_{\nu},\boldsymbol {\dot r}_1,\ldots,\boldsymbol {\dot r}_{\nu})=0,\quad k=1,\ldots, n,$$ which are called constraint equations or constraints. The forces $\boldsymbol R_i$ are called forces of constraint reactions or simply reactions. Reactions are to be determined during the solution of the problem. If there are no constraints, the reaction forces are set to zero.

Let us write Newton's second law for each particle:
\begin{equation}\label{dfgjjj}m_i\boldsymbol {\ddot r}_i=\boldsymbol F_i+\boldsymbol R_i,\quad i=1,\ldots, {\nu}.\end{equation}
Introduce the following notation:
\begin{align}G&=\mathrm{diag}\,(m_1,m_1,m_1,\ldots,m_{\nu},m_{\nu},m_{\nu});\quad x=(X_1,Y_1,Z_1,\ldots,X_{\nu},Y_{\nu},Z_{\nu})^T;\nonumber\\
f(t,x,\dot x)&=(F_1^x,F_1^y,F_1^z,\ldots ,F_{\nu}^x,F_{\nu}^y,F_{\nu}^z),\quad N(t,x,\dot x)=(R_1^x,R_1^y,R_1^z,\ldots ,R_{\nu}^x,R_{\nu}^y,R_{\nu}^z),\nonumber\\
\ph^k(t,x,\dot x)&=\Phi^k.\nonumber
\end{align}

\section{The D'Alembert-Lagrange Principle}\label{2}
Throughout what follows, all objects under consideration are assumed
to be $C^\infty$-smooth in their respective domains of definition.

Throughout what follows, we systematically employ row and column vectors that depend,
in general, on $x, \dot{x}$, and $t$.

When these objects depend solely on $x$ rather than on $\dot{x}$ and $t$, they admit a transparent geometric interpretation: a row vector represents the components of a rank-one covariant tensor on a manifold equipped with local coordinates $x$, whereas a column vector corresponds to a vector (i.e., a contravariant tensor). For further details, see \cite{nov}, \cite{stern}.

In the general case, where column and row vectors depend additionally on time and velocities, a simple geometric interpretation is absent. Nevertheless, the formulas remain valid and the objects are rigorously defined, as we demonstrate in what follows.

The row vector $f$ and column vector $\ph=(\ph^1,\ldots,\ph^n)^T(t,x,\dot x)$ are assumed to be smooth functions on the direct product $I\times D\times \mathbb{R}^m.$ Here $D\subset\mathbb{R}^m$ is a domain, $$ m=3\nu,\quad x=(x^1,\ldots, x^m)^T\in D,\quad t\in I=(t_1,t_2),\quad \dot x\in\mathbb{R}^m.$$
System (\ref{dfgjjj}) is written as follows:
\begin{equation}\label{seg44}
G\ddot x=f^T(t,x,\dot x)+N^T(t,x,\dot x),\end{equation}
and the constraints take the form:
\begin{equation}\label{xfg55}
\ph(t,x,\dot x)=0.\end{equation}
The domain $D$ is called the configuration space; the domain $D\times\mathbb{R}^m\ni (x,\dot x)$ is called the phase space; the domain $I\times D\times\mathbb{R}^m\ni (t,x,\dot x)$ is called the extended phase space of system (\ref{seg44}).

We represent system (\ref{seg44}) in covariant form:
\begin{equation}\label{azzxxx}[\mathscr T]=f+N,\quad[\mathscr T]= \frac{d}{dt}\frac{\partial\mathscr{T}}{\partial \dot x}-\frac{\partial\mathscr{T}}{\partial  x}, \end{equation}
where
$$\mathscr T=\frac{1}{2}\sum_{i=1}^\nu m_i|\boldsymbol {\dot r}_i|^2=\frac{1}{2}\dot x^TG\dot x$$
is the kinetic energy of the system.
\begin{rem}The representation of system (\ref{seg44}) in the form (\ref{azzxxx})
plays a key role in the following exposition (see section \ref{cbbbbb}). Therefore, we will adhere to this form, although at this stage it may seem strange.\end{rem}

If there exists a function $\mathscr V(t,x,\dot x)$ such that
$f=[\mathscr{V}],$
then the active forces are called generalized potential forces.
The function $\mathscr V(t,x,\dot x)$ is called a generalized potential.
If the generalized potential does not depend on $\dot x$, it is called a potential or potential energy.

Assume that the constraints are non-degenerate:
$$ \mathrm{rang}\,\ph_{\dot x}(t,x,\dot x)=n<m,\quad \forall (t,x,\dot x)\in I\times D\times \mathbb{R}^m,\qquad\ph_{\dot x}=\begin{pmatrix}
\frac{\partial \ph^1}{\partial \dot x^1}  & \cdots & \frac{\partial \ph^1}{\partial \dot x^m} \\
\vdots    & \ddots & \vdots  \\
\frac{\partial \ph^n}{\partial \dot x^1} & \cdots & \frac{\partial \ph^n}{\partial \dot x^m}
\end{pmatrix}.$$
The set of solutions to equation (\ref{xfg55})
$$S=\{(t,x,\dot x)\mid\ph(t,x,\dot x)=0\},\quad \dim S=2m+1-n,$$ which we assume to be non-empty, is a smooth submanifold of the extended phase space.

This fact is a consequence of the manifold $S$ being defined by $n$ independent smooth equations on the $(2m+1)$-dimensional space with coordinates $(t, x, \dot{x})$.

\begin{theo}\label{as11}
There exists a unique row vector $N=N(t,x,\dot x)$ smooth on $I\times D\times \mathbb{R}^m$ such that

1) the functions $\ph$ are first integrals of system (\ref{azzxxx});

2) for any column vector $\xi\in\mathbb{R}^m$ such that
\begin{equation}\label{sdg-0}\ph_{\dot x}(t,x,\dot x)\xi=0\end{equation}
the equality
\begin{equation}\label{sfg43}N(t,x,\dot x)\xi =0\end{equation}
holds.
\end{theo}
By Theorem \ref{as11}, the manifold $S$ is an invariant manifold of system (\ref{azzxxx}).
\begin{df}[\cite{rum}]The linear space of vectors $\xi\in\mathbb{R}^m$ satisfying (\ref{sdg-0}), i.e., $\ker\ph_{\dot x}(t,x,\dot x)$, is called the space of virtual displacements. This space is associated with the values of the variables $(t,x,\dot x)$.

The dimension of the space of virtual displacements is called the number of degrees of freedom of the system
(\ref{azzxxx}), (\ref{xfg55}).
\end{df} Obviously, the number of degrees of freedom is equal to $m-n=:r$.

\begin{df} If the forces $N$ are chosen in accordance with Theorem \ref{as11}, then the system (\ref{azzxxx}), (\ref{xfg55}) is called a system with ideal constraints. The constraints (\ref{xfg55}) are called ideal, and the forces $N$ are called reactions of ideal constraints.\end{df}

In what follows, we consider only systems with ideal constraints.

\subsubsection*{ Proof of Theorem \ref{as11}.} \begin{lem}[\cite{kostr}]\label{dfh6ll} Let $X,Y,Z$ be vector spaces and
$$A:X\to Y,\quad B:X\to Z$$ be linear operators.

Then if $\ker A\subset\ker B,$ there exists an operator $\Lambda:Y\to Z$ such that $B=\Lambda A$:
\[
  \begin{tikzcd}
    X \arrow{r}{A} \arrow[swap]{dr}{B} & Y \arrow{d}{\Lambda} \\
     & Z
  \end{tikzcd}
\]
\end{lem}
From this and formulas (\ref{sdg-0}), (\ref{sfg43}), it follows 
that there is a row vector $\Lambda=(\lambda_1,\ldots,\lambda_n)(t,x,\dot x)$ such that 
\begin{equation}\label{dfg54hh}N=\Lambda\ph_{\dot x}.\end{equation}
Let us differentiate $\ph$ along the solutions of system (\ref{azzxxx}) and equate the result to zero:
$$\ph_t+\ph_x\dot x+\ph_{\dot x}G^{-1}(f^T+N^T)=0.$$
Substituting $N$ from formula (\ref{dfg54hh}) here, we find
\begin{equation}\label{xdfgaxv}\Lambda^T=-\big(\ph_{\dot x}G^{-1}\ph_{\dot x}^T\big)^{-1}(\ph_t+\ph_x\dot x+\ph_{\dot x}G^{-1}f^T).\end{equation}
Here we use the following fact from linear algebra.
\begin{lem}[\cite{golub}]\label{xfg6zzzz6}If a square $m\times m$ matrix $P$ is symmetric and positive definite, and matrix $B$ consists of $m$ columns and $n$ rows ($m>n$), with $\mathrm{rang}\,B=n,$ then the matrix $BPB^T$ is symmetric and positive definite.\end{lem}

The theorem is proved.

Equation (\ref{azzxxx}) with the substituted reaction (\ref {dfg54hh}) is called the Lagrange equation with multipliers or the Lagrange equation of the first kind.

\begin{theo}\label{dfg44}
Let the function $x(t)$ satisfy (\ref {xfg55}) and 
\begin{align}\ph_{\dot x}&\big(t,x(t),\dot x(t)\big)\xi=0\Longrightarrow \nonumber\\
&([\mathscr T]-f)\xi=0.\label{xsdfg23ee}\end{align}
Then $x(t)$ satisfies (\ref{azzxxx}) with reactions (\ref{dfg54hh}), (\ref{xdfgaxv}).
\end{theo}Equation (\ref{xsdfg23ee}) with virtual displacements $\xi$ is called the general equation of dynamics.

If one sequentially substitutes the basis vectors of the space of virtual displacements into (\ref{xsdfg23ee}), a system of $r$ scalar differential equations is obtained, each containing components of the vector $\ddot x$. Together with the constraint equations, this yields a system of $m$ differential equations of order $2r+n=2m-n=\dim S-1$.

{\it Proof of Theorem \ref{dfg44}.} Indeed, by Lemma \ref{dfh6ll} there exists a row vector
$\tilde\Lambda(t)$ such that 
\begin{equation}\label{dfh6pp}
[\mathscr T]-f=\ddot x^T(t)G-f\big(t,x(t),\dot x(t)\big)=\tilde\Lambda(t)\ph_{\dot x}\big(t,x(t),\dot x(t)\big).\end{equation}
Let us differentiate the equation $\ph\big(t,x(t),\dot x(t)\big)=0:$
$$\ph_t\big(t,x(t),\dot x(t)\big)+\ph_x\big(t,x(t),\dot x(t)\big)\dot x(t)+\ph_{\dot x}\big(t,x(t),\dot x(t)\big)\ddot x(t)=0.$$
Substituting $\ddot x$ from (\ref{dfh6pp}) here, we verify that 
$\tilde \Lambda(t)=\Lambda\big(t,x(t),\dot x(t)\big).$

The theorem is proved.

The following theorem is obvious.
\begin{theo}\label{dfg4xde4}
Let $x(t)$ satisfy (\ref{azzxxx}) with reactions (\ref{dfg54hh}), (\ref{xdfgaxv}).
Then the following implication holds: $$\ph_{\dot x}\big(t,x(t),\dot x(t)\big)\xi=0\Longrightarrow
([\mathscr T]-f)\xi=0.$$\end{theo}

Theorems \ref{as11}, \ref{dfg44}, \ref{dfg4xde4} constitute the D'Alembert-Lagrange principle.

\subsection{The Appell–Hamel System}As an example of a system with a nonlinear nonholonomic constraint, consider the Appell–Hamel system, describing the motion of a particle of mass $m$ in the gravity field with Cartesian coordinates $(x^1, x^2, x^3)$ subject to the constraint\begin{equation}\label{sxdfgh456}\varphi
= {(\dot{x}^1)}^2 + {(\dot{x}^2)}^2 - b{(\dot{x}^3)}^2 = 0,\end{equation}where $b >0$ is a constant. A detailed description of a mechanical system in which this constraint arises naturally can be found in \cite{nem-fuf}.

We study the system in the domain $\dot{x}^3 > 0$ of the extended phase space. In other domains, the analysis is analogous.

The Lagrangian of this system has the form$$L = \frac{m}{2}\left({(\dot{x}^1)}^2 + {(\dot{x}^2)}^2 + {(\dot{x}^3)}^2\right)-mgx^3.$$Let us write out the d'Alembert–Lagrange principle:
\begin{equation}\label{dxfh790i98}[L]_{x^1}\delta x^1 + [L]_{x^2}\delta x^2 + [L]_{x^3}\delta x^3 = m\left(\ddot{x}^1\delta x^1 + \ddot{x}^2\delta x^2 + (\ddot{x}^3+g)\delta x^3\right) = 0,\end{equation}with virtual displacements $(\delta x^1,\delta x^2,\delta x^3)$ specified by the equation$$\frac{\partial \varphi}{\partial \dot{x}^1}\delta x^1 + \frac{\partial \varphi}{\partial \dot{x}^2}\delta x^2 + \frac{\partial \varphi}{\partial \dot{x}^3}\delta x^3 = 2(\dot{x}^1\delta x^1 + \dot{x}^2\delta x^2 - b\dot{x}^3\delta x^3) = 0.$$
Let us express
$$\delta x^3=\frac{\dot x^1\delta x^1+\dot x^2\delta x^2}{b\dot x^3}$$
from the last equation and
substituting it into (\ref{dxfh790i98}), we obtain,
due to the independence of $\delta x^1$ and $\delta x^2$:
\begin{equation}\label{sdfhbg0976}\ddot x^1+\frac{(\ddot x^3+g)\dot x^1}{b\dot x^3}=0,\quad \ddot x^2+\frac{(\ddot x^3+g)\dot x^2}{b\dot x^3}=0.\end{equation}
 Equations (\ref{sdfhbg0976}), together with the constraint equation (\ref{sxdfgh456}), form a complete system of differential equations for the functions $x^1(t), x^2(t), x^3(t)$.

Let us return to the general case of equation (\ref{xsdfg23ee}).

We define the total order of the system as the sum of the orders of its independent differential equations. This quantity characterizes the dimension of the space of admissible initial data (at a fixed initial time), provided the system is regular.

If one sequentially substitutes the basis vectors of the space of virtual displacements into (\ref{xsdfg23ee}), a system of $r$ scalar differential equations is obtained, each containing components of the vector $\ddot x$. Together with the constraint equations, this yields a system of $m$ differential equations of total order  $2r+n=2m-n=\dim S-1$.

In the Appell–Hamel example, $x=(x^1,x^2,x^3)^T$, $m=3$, and $n=1$, with
$$S=\{(t,x,\dot x)\mid {(\dot{x}^1)}^2 + {(\dot{x}^2)}^2 - b{(\dot{x}^3)}^2 = 0,\quad \dot x^3>0\},\quad\dim S=6.$$
We have a pair of second-order equations (\ref{sdfhbg0976}) and one first-order equation (\ref{sxdfgh456}).

The total order of the system  \eqref{sdfhbg0976}, \eqref{sxdfgh456} is $2+2+1=5$.

\section{On the Correctness of the Definition of Reactions and Virtual Displacements}\label{sdf55}
The definitions of the reactions $N$ and the space of virtual displacements involve the functions $\ph$.
In this section, we show that neither the reactions nor the space of virtual displacements depend on the choice of the functions $\ph$ that define the manifold $S$. In other words, the reactions and virtual displacements depend on the geometry of the manifold $S$, rather than on the specific way it is analytically specified.

Suppose that the manifold $S$ is also defined by a function
$$\psi(t,x,\dot x)=(\psi^1,\ldots,\psi^n)^T,\quad \mathrm{rank}\,\psi_{\dot x}=n<m,\quad
\forall (t,x,\dot x)\in I\times D\times \mathbb{R}^m,$$
i.e.,
$$S=\{\ph=0\}=\{\psi=0\}.$$

In this case, every point of the manifold $S$  has a neighborhood in the extended phase space in which a smooth function
$$U=(U^1,\ldots,U^n)^T(t,x,\dot x,z)$$
is defined, where the parameter $z$ varies in a neighborhood of zero in $\mathbb{R}^n$.

The function $U$ has the following properties (apply Theorem \ref{sdfgb09} with $w=(t,x,\dot x)$):
$$U(t,x,\dot x,z)=0\Leftrightarrow z=0,\quad \det U_z(t,x,\dot x,0)\ne 0,\quad
\psi(t,x,\dot x)=U(t,x,\dot x,\ph(t,x,\dot x)).$$
The following formulas hold:
\begin{equation}\label{sdgb77890}\psi_t=U_t(t,x,\dot x,\ph)+U_z(t,x,\dot x,\ph)\ph_t,\quad U_t(t,x,\dot x,\ph)\mid_{\ph=0}=0.\end{equation}
Similarly,
\begin{equation}\label{sdgb77890-1}\psi_x\mid_S=U_z\ph_x\mid_S,\quad \psi_{\dot x}\mid_S=U_z\ph_{\dot x}\mid_S.\end{equation}

In particular, $$\mathrm{rank}\,\big(\psi_{\dot x}\mid_S\big)=n, \quad \ker \big(\psi_{\dot x}\mid_S\big)=\ker\big( \ph_{\dot x}\mid_S\big).$$
The last formula shows that the space of virtual displacements (which was originally given as the kernel of $\varphi_{\dot x}$) remains unchanged if $\varphi$ is replaced by $\psi$.

We compute the reactions corresponding to the different ways of specifying the manifold $S$:
\begin{align}
N^T&=-\ph^T_{\dot x}\big(\ph_{\dot x}G^{-1}\ph_{\dot x}^T\big)^{-1}(\ph_t+\ph_x\dot x+\ph_{\dot x}G^{-1}f^T)\nonumber\\
\tilde N^T&=-\psi^T_{\dot x}\big(\psi_{\dot x}G^{-1}\psi_{\dot x}^T\big)^{-1}(\psi_t+\psi_x\dot x+\psi_{\dot x}G^{-1}f^T).\label{zdv33}
\end{align}
Substituting equations (\ref{sdgb77890}) and (\ref{sdgb77890-1}) into the  equation (\ref{zdv33}), we obtain
$$
N^T\mid_S=\tilde N^T\mid_S.$$This formula shows that the reactions do not depend on the function ($\varphi$ or $\psi$) defining the constraint.

\section{Holonomic and non-holonomic constraints}\label{5}
\begin{df}Constraints (\ref{xfg55}) are called holonomic if there are functions
$$g=(g^1,\ldots, g^n)^T(t,x),\quad \mathrm{rang}\, g_x=n$$ such that 
$$S=\Big\{(t,x,\dot x)\mid\frac{d}{dt}g(t,x) =0\Big\},\quad \frac{d}{dt}g(t,x)=g_t(t,x)+g_x(t,x)\dot x.$$

Constraints that are not holonomic are called non-holonomic.
\end{df}
The holonomicity of constraints is equivalent to the integrability of a differential system on a manifold \cite{stern}.

The equations $g(t,x)=0$ are called geometric constraints, and the equations (\ref{xfg55}) are called differential constraints.

Hereafter, until the end of Section \ref{cbbbbb}, we will consider the constraints holonomic and write
\begin{equation}\label{xdfg11gy}\ph(t,x,\dot x)=g_t(t,x)+g_x(t,x)\dot x=0.\end{equation}
In particular,
\begin{equation}\label{xf11q}\ph_{\dot x}=g_x,\quad\ker \ph_{\dot x}(t,x,\dot x)=\ker g_x(t,x).\end{equation}

Assuming that $f$ is a potential force, note that passing to holonomic systems fundamentally changes the situation regarding the variational nature of the problem. In the case of holonomic systems, the standard conditional extremum problem for the action functional leads to the Lagrange equations with multipliers (\ref{azzxxx}), (\ref{dfg54hh}).

Whereas for non-holonomic systems, this variational problem leads to entirely different equations, namely the equations of vaconomic mechanics (\cite{arnold2006mathematical}, \cite{bloch2015nonholonomic}).

\section{Covariance of the variational derivative}\label{6}
Assume that the manifold
$$\Sigma_t=\{x\in D\mid g(t,x)=0\},\quad \dim\Sigma_t=r$$ is the image of a smooth manifold $Y,\quad \dim Y=r$ under the embedding\footnote{By an embedding we mean a smooth injective mapping
$u(t,\cdot):Y\to\mathbb{R}^m$ such that for every point $y\in Y$ and its image $x=u(t,y)\in \mathbb{R}^m$ there exist neighborhoods $U_y\subset Y$ and $U_x\subset\mathbb{R}^m$ with local coordinates $(y^1,\ldots,y^r)$ and $(x^1,\ldots,x^m)$ respectively, in which $u$ takes the form: $x^i=y^i,\quad i=1,\ldots, r,\quad x^s=0,\quad s=r+1,\ldots,m,$ moreover $U_x\cap \Sigma_t=\{x^s=0,\quad s=r+1,\ldots,m\}.$
The specified local coordinates smoothly depend on $t\in I$.}
$$u(t,\cdot):Y\to D,\quad \Sigma_t=u(t,Y).$$
Let $y=(y^1,\ldots,y^r)^T$ denote the local coordinates in $Y$;
\begin{equation}\label{fg11}g(t,u(t,y))=0,\quad u=(u^1,\ldots,u^m)^T, \quad x=u(t,y).\end{equation}
The linear mapping
$$u_y(t,y):T_yY\to T_{u(t,y)}\Sigma_t$$ is an isomorphism, in particular,
$\mathrm{rang}\,u_y=r$,
and the tangent space\footnote{The necessary definitions and a brief introduction to this and other differential-geometric objects can be found in \cite{nov}.}
$T_{u(t,y)}\Sigma_t$ coincides with the space of virtual displacements:
$$T_{u(t,y)}\Sigma_t=\ker \big(\ph_{\dot x}\mid_{x=u(t,y)}\big).$$
This is clear if we differentiate formula (\ref{fg11}) with respect to $y$ and use (\ref{xf11q}):
$$g_x(t,u(t,y))u_y(t,y)=0.$$
Consider a smooth function $\mathscr{F}(t,x,\dot x)$ defined on the extended phase space.
\begin{df}
The variational derivative of $\mathscr{F}$ is defined as the following differential operator
$$[\mathscr{F}]=([\mathscr{F}]_{1},\ldots ,[\mathscr{F}]_{m}),\quad [\mathscr{F}]_{k}=\frac{d}{dt}\frac{\partial \mathscr{F}}{\partial\dot x^k}-\frac{\partial \mathscr{F}}{\partial x^k}.$$
\end{df}
Direct calculation confirms that the variational derivative is a linear operation and for any function $w=w(t,x),\quad w:I\times D\to \mathbb{R}$
the identity 
$$\Big[\frac{dw}{dt}\Big]=0$$
holds.

\begin{theo}\label{dfggg456}
The equality
$$[\mathscr{F}]\Big|_{x=u(t,y)}u_y(t,y)=[F]$$
holds, where the function $F:I\times TY\to\mathbb{R}$ ($TY$ is the tangent bundle) is defined by the formula
\begin{equation}\label{dff}F(t,y,\dot y)=\mathscr{F}\Big|_{x=u(t,y)}=\mathscr{F}\big(t,u(t,y),u_t(t,y)+u_y(t,y)\dot y\big).\end{equation}
\end{theo}

{\it Proof of Theorem \ref{dfggg456}.} Let us carry out the calculations in coordinates. By formula (\ref{dff}) we have:
$$\frac{\partial F}{\partial y^i}=\frac{\partial \mathscr{F}}{\partial x^p}\frac{\partial u^p}{\partial y^i}+
\frac{\partial \mathscr{F}}{\partial \dot x^p}\Big(\frac{\partial^2 u^p}{\partial y^l\partial y^i}\dot y^l+\frac{\partial^2u^p}{\partial t\partial y^i}\Big);$$
and
$$
\frac{d}{dt}\frac{\partial F}{\partial \dot y^i}=\frac{d}{dt}\Big(\frac{\partial \mathscr{F}}{\partial \dot x^p}
\frac{\partial u^p}{\partial y^i}\Big)=\frac{\partial u^p}{\partial y^i}\Big(\frac{d}{dt}\frac{\partial \mathscr{F}}{\partial \dot x^p}\Big)+\frac{\partial \mathscr{F}}{\partial \dot x^p}\Big(\frac{\partial^2 u^p}{\partial y^l\partial y^i}\dot y^l+\frac{\partial^2u^p}{\partial t\partial y^i}\Big).$$
Subtracting one equality from the other, we obtain the required result.
The theorem is proved.

 Theorem \ref{dfggg456}  shows that the variational derivative transforms as a covector under the time-dependent change of configuration variables $x=u(t,y)$ induced by the embedding. The corresponding transformation of velocities is given by $\dot x=u_t+u_y\dot y$, while the covector transformation itself is governed by the Jacobian $u_y$. Thus, the variational derivative has an intrinsic geometric meaning, independent of the particular coordinates used to describe the constraint manifold $\Sigma_t$.

\section{Lagrange equations}\label{cbbbbb}
Theorem \ref{dfggg456} means that the variational derivative behaves like a covector field with respect to the mapping $u$.
Roughly speaking, the derivation of the Lagrange equations consists of using the mapping $u$ to perform a pullback of the general equation of dynamics from the domain $D$ to the manifold $Y$.

Let us proceed to rigorous formulations.

\begin{df}Generalized forces are the components of the row vector
\begin{equation}\label{sfg40}Q(t,y,\dot y):=f\big(t,u(t,y),u_t(t,y)+u_y(t,y)\dot y\big) u_y(t,y).\end{equation}
\end{df}
\begin{rem} Suppose that the covector (row vector) $f = (f_1, \ldots, f_m)$ is independent of $\dot{x}$ and depends solely on $t$ and $x$. Then formula (\ref{sfg40}) admits a transparent geometric interpretation: it is literally the pullback of the $1$-form $\omega = f_i dx^i$ to the manifold $\Sigma_t$, where time $t$ serves as a parameter. The result of this pullback is the covector $Q(t,y) = u^*\omega,\quad Q=Q_kdy^k$.\end{rem}

\begin{theo}\label{xsfgtyuio}
Let the function $y(t)$ be a solution of the Lagrange equations (of the second kind)
\begin{equation}\label{sdfgqas}[T]=Q,\end{equation}
where the kinetic energy $T=T(t,y,\dot y)$ is defined by the formula:
\begin{equation}\label{xdfsgggtyu}
T=\mathscr T\mid_{x=u(t,y)}=
\frac{1}{2}\big(u_t+u_y\dot y\big)^TG\big(u_t+u_y\dot y\big).
\end{equation}

Then the function $x(t)=u(t,y(t))$ satisfies the general equation of dynamics (\ref{xsdfg23ee}) with constraints (\ref{xdfg11gy}).\end{theo}
{\it Proof of Theorem \ref{xsfgtyuio}.} By Theorem \ref{dfggg456}, we have
\begin{equation}\label{xdfg6yiii}([\mathscr{T}]-f)\Big|_{x=u(t,y(t))}u_y(t,y(t))=([T]-Q)\Big|_{y=y(t)}.\end{equation}
The right side of this equality is zero by the assumption. As already noted, the image of the operator $u_y$ coincides with the space of virtual displacements, so the general equation of dynamics is satisfied.

The theorem is proved.

The manifold $Y$ is called the configuration manifold, and the local coordinates $y$ are called generalized coordinates of the system (\ref{sdfgqas}).

The manifolds $TY,\quad I\times TY$ (see \cite{nov} for details) are called the phase space and the extended phase space, respectively.

The components of the vector $\dot y$ are called generalized velocities.

\begin{theo}\label{dfb00dd}
Let the function $x(t)$ satisfy the general equation of dynamics (\ref{xsdfg23ee}) with constraints (\ref{xdfg11gy}), and at some $t_0\in I$ the equality
$g(t_0,x(t_0))=0$ holds.

Then there exists a unique solution $y(t)$ of equations (\ref{sdfgqas}) such that $$x(t)=u(t,y(t)).$$\end{theo}
{\it Proof of Theorem \ref{dfb00dd}}.
We integrate the equality $$\frac{d}{dt}g(t,x(t))=0$$ from $t_0$ to $t:\quad g(t,x(t))-g(t_0,x(t_0))=0.$
Thus, $x(t)\in \Sigma_t,\quad t\in I$.

For any $t$, the equation \begin{equation}\label{dh0009}x(t)=u(t,y)\end{equation} is uniquely solvable for $y$, since $u(t,\cdot):Y\to\Sigma_t$ is a diffeomorphism. Let us denote this solution as $y(t),\quad x(t)=u(t,y(t))$.

We will show that $y(t)$ is a smooth function. Indeed, (\ref{dh0009}) is a system of $m$ equations with an $r$-dimensional unknown vector $y$. We select $r$ independent equations from this system and apply the implicit function theorem regarding the smoothness of the solution with respect to the parameter.

The left side of equality (\ref{xdfg6yiii}) is zero by the assumption.
The theorem is proved.

\begin{rem}
The above argument leads to the following observation. The function $g(t,x)$ is a first integral of the system (\ref{azzxxx}) restricted to the manifold $S$. In other words, $g$ is a conditional first integral of the system (\ref{azzxxx}).
\end{rem}

\begin{theo}\label{asdf4eew}The formula
$T=T_2+T_1+T_0$ holds, where $$T_2=\frac{1}{2}\dot y^Tu_y^TGu_y\dot y$$ is a positive definite quadratic form of the generalized velocities; $T_1=u_t^TGu_y\dot y$ is a linear form, $T_0=u_t^TGu_t/2$.\end{theo}
This statement follows directly from formula (\ref{xdfsgggtyu}) and Lemma \ref{xfg6zzzz6}.

From Theorem (\ref{asdf4eew}) it follows that system (\ref{sdfgqas}) can be represented in normal form:
$$\ddot y=a(t,y,\dot y),$$ and therefore the Cauchy existence and uniqueness theorem is valid for it.

If the forces $f$ are generalized potential forces, then the forces $Q$ are also generalized potential forces -- due to the covariance of the operation $[\cdot]$. Moreover
$Q=[V],\quad V(t,y,\dot y)=\mathscr V\mid_{x=u(t,y)},$ and equations (\ref{sdfgqas}) turn into Lagrange equations:
$$[L]=0,\quad L=T-V.$$The function $L=L(t,y,\dot y)$ is called the Lagrange function or Lagrangian.

\section{On classical notations}\label{8}
Traditionally, in mechanics, the vector of virtual displacements $\xi$ is written as follows:
$$\xi=(\delta \boldsymbol r_1,\ldots,\delta\boldsymbol  r_\nu)^T.$$
Accordingly, equations (\ref{sdg-0}) and (\ref{sfg43}) have the form
$$\sum_{i=1}^\nu\Big(\frac{\partial \Phi^k}{\partial\boldsymbol{\dot r}_i},\delta \boldsymbol r_i\Big)=0,\quad\sum_{i=1}^\nu(\boldsymbol R_i,\delta \boldsymbol r_i)=0,\quad k=1,\ldots,n.$$

The general equation of dynamics (\ref{xsdfg23ee}) has the form
$$\sum_{i=1}^\nu(m_i\boldsymbol{\ddot r}_i-\boldsymbol F_i,\delta\boldsymbol r_i)=0.$$

The embedding $u$ is introduced as follows:
$$u(t,\cdot):y\mapsto (\boldsymbol r_1(t,y),\ldots,\boldsymbol r_\nu(t,y)).$$
Accordingly, the generalized forces (\ref{sfg40})
are calculated by the formula:
$$Q_j=\sum_{i=1}^\nu\Big(\bs F_i,\frac {\partial\bs r_i}{\partial  y^j}\Big),\quad Q=(Q_1,\ldots,Q_r).$$

\section{Appendix}\label{9}

Consider smooth functions (column vectors)
$$a, b: F \to \mathbb{R}^n, \quad n < p$$
defined on the domain $F \subset \mathbb{R}^p = \{w = (w^1, \ldots, w^p)^T\}$.

Suppose that the following condition holds at all points of the domain $F$:
$$\mathrm{rank}\, a_w = \mathrm{rank}\, b_w = n.$$
Furthermore, assume that the smooth manifolds
$\{w \in F \mid a(w) = 0\}$ and $\{w \in F \mid b(w) = 0\}$ coincide. Let us denote this $(p-n)$-dimensional manifold by $M,\quad M\subset\mathbb{R}^p$.

\begin{theo}\label{sdfgb09}Fix an arbitrary point $w_0 \in M$. Then there exists an open neighborhood $W_{w_0}\subset F$ of this point, an open neighborhood $Z$ of the origin in the space
$\mathbb{R}^n=\{z=(z^1,\ldots,z^n)\}$, and a smooth function
$$U: W_{w_0} \times Z \to \mathbb{R}^n, \quad U = U(w, z),$$
such that:
\begin{enumerate}[\rm 1)]
    \item $U(w, z) = 0 \iff z = 0$;
    \item $\det U_z(w, 0) \neq 0$;
    \item $b(w) = U(w, a(w))$.
\end{enumerate}
\end{theo}

{\it Proof of Theorem \ref{sdfgb09}.}
Let us introduce the notation $u = (w^1, \ldots, w^n)^T$ and $v = (w^{n+1}, \ldots, w^p)^T$.

Without loss of generality, we may assume that the inequality
\begin{equation}\label{sfh098}
\det a_u(u, v) \neq 0
\end{equation}
holds in a neighborhood of the point $w_0$.
By the implicit function theorem, we have:
\begin{equation}\label{sdf0}
z = a(u, v) \iff u = f(z, v), \quad z = a(f(z, v), v), \quad u = f(a(u, v), v).
\end{equation}
Note that $M \cap W_{w_0} = \{w = (u,v) \in W_{w_0} \mid z = a(u, v) = 0\}$.

Define $U(w, z) := b(f(z, v), v)$.

Property 3) follows directly from \eqref{sdf0}.

We now verify property 1). Suppose $U(w, z) = 0$. Then $(f(z, v), v) \in M$, which, by \eqref{sdf0}, implies $z = 0$. The reverse implication follows similarly.

Next, we verify property 2). Observe that for every point $w \in M$, the kernels of the differentials of the mappings $a$ and $b$ coincide with the tangent space to the manifold $M$ at the given point:
$$\ker a_w(w) = \ker b_w(w) = T_w M.$$
By Lemma \ref{dfh6ll} there exists a linear operator $C(w): \mathbb{R}^n \to \mathbb{R}^n$ such that
$$b_w(w) = C(w) a_w(w).$$
Since the operators $a_w, b_w: \mathbb{R}^p \to \mathbb{R}^n$ are surjective, we have $\det C(w) \neq 0$.

We express the corresponding Jacobian matrices in block form:
$$a_w = (a_u, a_v), \quad b_w = (b_u, b_v).$$
Hence, $b_u = C a_u$. Thus, taking relation \eqref{sfh098} into account, we obtain
$$\det b_u(u, v) \neq 0 \quad \text{for all } w = (u, v) \in M \cap W_{w_0}.$$

Using the definition of the function $U$, we compute the derivative with respect to $z$:
$$U_z(w, 0) = b_u(f(0, v), v) \, f_z(0, v),$$
where $(f(0, v), v) \in M$.

It follows from relations \eqref{sdf0} that
$$\mathrm{id}_{\mathbb{R}^n} = a_u(f(0, v), v) \, f_z(0, v).$$
Consequently, $\det f_z(0, v) \neq 0$, which yields $\det U_z(w, 0) \neq 0$.

This proves property 2).

The proof of the theorem is now complete.

\end{document}